\pdfoutput=1
\documentclass[11pt]{amsart}

\usepackage{graphicx}
\usepackage[english]{babel}
\usepackage[T1]{fontenc}
\usepackage[latin1]{inputenc}
\usepackage{amsfonts}
\usepackage{amssymb}
\usepackage{amsthm}
\usepackage{amsmath}
\usepackage{tikz}
\usepackage{float}
\usepackage{enumerate}
\usepackage{accents}
\usepackage{mathtools}
\usepackage{mathrsfs}
\usepackage{comment}
\usepackage{afterpage}
\usepackage{bbm}
\usepackage{dsfont}
\usepackage{stmaryrd}
\usepackage{hyperref}
\usepackage{mleftright}
\usepackage{subfig}
\usepackage{soul}
\usepackage{tikz-cd} 
\usepackage{fullpage}
\usepackage{cleveref}

\numberwithin{equation}{section}

\theoremstyle{plain}
\newtheorem{thm}{Theorem}[section]
\newtheorem*{thm*}{Theorem}
\newtheorem{prop}[thm]{Proposition}
\newtheorem*{prop*}{Proposition}
\newtheorem{cor}[thm]{Corollary}
\newtheorem*{cor*}{Corollary}
\newtheorem{lem}[thm]{Lemma}

\theoremstyle{definition}
\newtheorem{defn}[thm]{Definition}
\newtheorem*{defn*}{Definition}
\newtheorem{ex}[thm]{Example}
\newtheorem{rmk}[thm]{Remark}
\newtheorem*{rmk*}{Remarks}

\newtheorem*{conj*}{Conjecture}
\newtheorem{quest}[thm]{Question}
\newtheorem*{quest*}{Question}

\newtheoremstyle{blue-environment}{}{}{}{}{\color{blue}\bfseries}{.}{ }{}
\theoremstyle{blue-environment}

\newcommand{\acts}{\curvearrowright}
\newcommand{\ra}{\rightarrow}

\newcommand{\sq}{\subseteq}

\newcommand{\x}{\times}
\renewcommand{\o}{\circ}
\newcommand{\id}{\mathrm{id}}

\newcommand{\mc}{\mathcal}
\newcommand{\mf}{\mathfrak}

\newcommand{\R}{\mathbb{R}}
\newcommand{\Z}{\mathbb{Z}}
\newcommand{\N}{\mathbb{N}}

\newcommand{\C}{\mathbb{C}}

\newcommand{\s}{\sigma}
\newcommand{\eps}{\epsilon}
\newcommand{\Om}{\Omega}

\newcommand{\g}{\gamma}

\newcommand{\CAT}{{\rm CAT(0)}}
\newcommand{\Out}{{\rm Out}}
\newcommand{\Aut}{{\rm Aut}}

\begin{document}

\title{Automorphism growth and group decompositions} 

\author[E.\,Fioravanti]{Elia Fioravanti}\address{Institute of Algebra and Geometry, Karlsruhe Institute of Technology}\email{elia.fioravanti@kit.edu} 
\thanks{The author is supported by Emmy Noether grant 515507199 of the Deutsche Forschungsgemeinschaft (DFG)}

\begin{abstract}
    Let $G$ be a finitely generated group with an automorphism $\varphi\in\Aut(G)$, or an outer automorphism $\phi\in\Out(G)$. Suppose that $G$ decomposes into simpler pieces on which the growth behaviour of $\varphi$ and $\phi$ is known, particularly as a direct product, free product, or graph of groups. This article is devoted to the (often not entirely straightforward) problem of deducing information about the growth rates of $\varphi$ and $\phi$ on the whole $G$.
\end{abstract}

\maketitle

\section{Introduction}

Let $G$ be a finitely generated group. Fix a finite generating set $S\sq G$ and denote by $|g|$ the resulting \emph{word length} of an element $g\in G$. We can similarly consider the \emph{conjugacy length}
\[ \|g\|:=\min_{h\in G} |hgh^{-1}| .\]
Given an automorphism $\varphi\in\Aut(G)$ and an outer automorphism $\phi\in\Out(G)$, it is natural to wonder about the speed of growth of the sequences 
\begin{align*}
    n&\mapsto |\varphi^n(g)|, & n&\mapsto\|\phi^n(g)\| .
\end{align*}
We consider such sequences up to bi-Lipschitz equivalence, meaning that we declare $a_n\sim b_n$, for two sequences $a_n,b_n$, if there exists a constant $C>0$ such that $a_n\leq Cb_n$ and $b_n\leq Ca_n$ for all $n\in\N$. Up to this equivalence, the sequences $n\mapsto |\varphi^n(g)|$ and $n\mapsto\|\phi^n(g)\|$ are completely independent of the choice of the finite generating set of $G$ used to define $|\cdot|$ and $\|\cdot\|$, and we refer to them as the \emph{growth rates} of $g$ under $\varphi$ and $\phi$, respectively.

Given a finitely generated group $G$ with automorphisms $\varphi\in\Aut(G)$ and $\phi\in\Out(G)$, completely describing growth rates is often a difficult problem. Already for the free group $F_n$, this problem is inapproachable by elementary techniques, and it was solved only relatively recently \cite{Levitt-GAFA} using the refined train-track technology developed in \cite{BH92,BFH1,BFH2,Bridson-Groves}. A complete description of growth rates for automorphisms of hyperbolic groups was recently announced in \cite{CHHL}, also using JSJ decompositions \cite{RS97,GL-JSJ}. There, all growth rates are bi-Lipschitz equivalent to sequences of the form $n\mapsto n^p\lambda^n$ with $p\in\N$ and $\lambda\geq 1$ a Perron number.

At the same time, there are also finitely generated groups with automorphisms displaying rather exotic growth behaviours; see for instance \cite{Coulon} and \Cref{ex:exotic} below. Very little is obvious about the general properties of growth rates, and countless natural questions remain about them. We list a few in \Cref{quest}.

In this article, we consider a finitely generated group with an automorphism $\varphi\in\Aut(G)$ and its outer class $\phi:=[\varphi]\in\Out(G)$, in each of the following three cases:
\begin{enumerate}
    \item $G$ splits as $G_1\x\dots\x G_k\x\Z^m$ for centreless, directly indecomposable groups $G_i$ (\Cref{sect:direct_products});
    \item $G$ splits as a $\phi$--invariant graph of groups with undistorted vertex groups $G_i$ (\Cref{subsec:invariant_GOGs});
    \item $G$ splits as a free product $G_1\ast\dots\ast G_k\ast F_m$ for freely indecomposable groups $G_i$ (\Cref{subsec:free_products}).
\end{enumerate}
In each of these cases, we give a description of the growth rates of $\varphi$ and $\phi$ on $G$, assuming that enough is known about the growth rates of the restrictions of $\varphi$ and $\phi$ to the $G_i$. It should not be surprising that such descriptions are possible, but the details of this are not always straightforward, albeit based on standard techniques (particularly, train tracks in the free product case).

The main results of the article are \Cref{cor:growth_direct_product}, \Cref{prop:growth_GOGs} and \Cref{prop:growth_estimate_fully_irreducible}, corresponding to each of the three situations described above. These are all used in the author's work on growth rates of automorphisms of right-angled Artin groups, right-angled Coxeter groups and, more generally, compact special groups \cite{Fio11}.

\smallskip
{\bf Acknowledgements.} 
I am grateful to Martin Bridson, R\'emi Coulon, Camille Horbez and Ric Wade for helpful conversations related to the contents of this article.

\section{Growth rates}\label{sect:growth}

\subsection{General properties}\label{sub:general_growth}

If $a,b\colon\N\ra\R_{>0}$ are two sequences, we write $a\preceq b$ if there exists a constant $C>0$ such that $a_n\leq Cb_n$ for all $n\geq 0$. We say that $a$ and $b$ are \emph{equivalent}, written $a\sim b$, if we have both $a\preceq b$ and $b\preceq a$. We denote by $[1]$ the equivalence class of constant sequences.

\begin{defn}
    A \emph{growth rate} is a $\sim$--equivalence class $[a_n]$ of sequences in $(\R_{>0})^{\N}$ with $[a_n]\succeq [1]$. We denote by $(\mf{G},\preceq)$ the set of growth rates with the poset structure induced by the relation $\preceq$.
\end{defn}

Given two growth rates $[a_n],[b_n]\in\mf{G}$, the sum $[a_n]+[b_n]:=[a_n+b_n]$ is a well-defined growth rate. We will often simply write $a_n\preceq b_n$ and $a_n\sim b_n$, omitting the square brackets when this streamlines notation without causing ambiguities. 

Let $G$ be a group with a finite generating set $S$. We denote by $|\cdot|_S$ and $\|\cdot\|_S$ the \emph{word length} and \emph{conjugacy length} on $G$ associated to $S$, as defined in the Introduction. If $T$ is a different finite generating set, then $|\cdot|_T$ and $\|\cdot\|_T$ are bi-Lipschitz equivalent to $|\cdot|_S$ and $\|\cdot\|_S$, respectively. For this reason, we will simply write $|\cdot|$ and $\|\cdot\|$ from now on, as the choice of generating set will play no role. Occasionally, in the presence of two finitely generated groups $G$ and $H$, we will write $|\cdot|_G$ and $|\cdot|_H$ to distinguish between the two length notions.

A finitely generated subgroup $H\leq G$ is \emph{undistorted} if the inclusion $(H,|\cdot|_H)\hookrightarrow (G,|\cdot|_G)$ is bi-Lipschitz. Similarly, we say that $H$ is \emph{conjugacy-undistorted} if the inclusion $(H,\|\cdot\|_H)\hookrightarrow (G,\|\cdot\|_G)$ is bi-Lipschitz. In general, undistortion alone does not suffice to deduce conjugacy-undistortion, and one rather needs some form of convexity.

\begin{ex}\label{ex:conj_undist}
The following are straightforward observations.
    \begin{enumerate}
        \item Quasi-convex subgroups of hyperbolic groups are undistorted and conjugacy-undistorted.
        \item Let $G\acts X$ be a proper cocompact action on a $\CAT$ space. If $H\leq G$ acts cocompactly on a convex subspace $Y\sq X$, then $H$ is undistorted and conjugacy-undistorted in $G$.
    \end{enumerate}
\end{ex}

Consider now $\varphi\in\Aut(G)$ and its outer class $\phi\in\Out(G)$. The following is classical.

\begin{lem}\label{lem:aut_bilip}
    The map $\varphi\colon G\ra G$ is bi-Lipschitz with respect to both $|\cdot|$ and $\|\cdot\|$.
\end{lem}

As explained in the Introduction, $\varphi$ and $\psi$ attach a growth rate to each element of $G$.

\begin{defn}
    The \emph{growth rate} of an element $g\in G\setminus\{1\}$ under $\varphi$ is the $\sim$--equivalence class $\big[|\varphi^n(g)|\big]$ in $\mf{G}$. Similarly, the \emph{growth rate} of $g$ under $\phi$ is\footnote{Note that $\phi^n(g)$ is not a well-defined element of $G$, but it is a well-defined conjugacy class.} the equivalence class $\big[\|\phi^n(g)\|\big]\in\mf{G}$. 
\end{defn}

Since $\|\cdot\|\leq |\cdot|$, we always have $\big[\|\phi^n(g)\|\big]\preceq \big[|\varphi^n(g)|\big]$. Moreover, $\big[|\varphi^n(g)|\big]\sim [1]$ holds if and only if a power of $\varphi$ fixes $g$, and $\big[\|\phi^n(g)\|\big]\sim[1]$ holds if and only if a power of $\phi$ preserves the conjugacy class of $g$. For this reason, it makes sense to artificially define the growth rate of the identity of $G$ to be $[1]$, even though the identity has length $0$.

We denote by $\mc{G}(\varphi)\sq\mf{G}$ and $\mf{g}(\phi)\sq\mf{G}$ (or $\mc{G}(G,\varphi)$ and $\mf{g}(G,\phi)$ if there is any ambiguity) the sets of all growth rates of $\varphi$ and $\phi$, as $g$ varies in $G$. In keeping with the above conventions, we will generally denote growth rates of automorphisms by the letter $\mc{O}$, and growth rates of outer automorphisms by the letter $\mf{o}$.

\begin{rmk}\label{rmk:restriction_new}
    Consider a subgroup $H\leq G$, and let $\phi\in\Out(G)$ preserve the $G$--conjugacy class of $H$. We can define a ``restriction'' $\phi|_H\in\Out(H)$ by choosing a representative $\varphi\in\Aut(G)$ of $\phi$ with $\varphi(H)=H$, and then considering the outer class $[\varphi|_H]\in\Out(H)$. However, the restriction $\phi|_H$ is not uniquely defined: denoting by $N_G(H)$ the normaliser of $H$, the conjugation action $N_G(H)\acts H$ determines a subgroup $C^G_H\leq\Out(H)$, and it is only the coset $\phi|_H\cdot C^G_H$ that is unique.

    Nevertheless, if $H\leq G$ is conjugacy-undistorted, then the set of growth rates $\mf{g}(\phi|_H)$ is well-defined. Indeed, any two possible restrictions $\phi|_H\in\Out(H)$ differ by the restriction to $H$ of an inner automorphism of $G$, and so conjugacy lengths grow at the same speed under their powers. (Since $H$ is conjugacy-undistorted, it does not matter whether we use $\|\cdot\|_H$ or $\|\cdot\|_G$.)
\end{rmk}

There are two additional elements of $\mf{G}$ that we can associate with $\varphi$ and $\phi$. In some sense, they play the role of a ``maximum'' for the sets $\mc{G}(\varphi)$ and $\mf{g}(\phi)$, but it is important to stress that, a priori, they do {\bf not} lie in $\mc{G}(\varphi)$ or $\mf{g}(\phi)$. Fixing any finite generating set $S\sq G$, we write:
\begin{align*}
    \overline{\mc{O}}_{\rm top}(\varphi)&:=\big[\s_S(\varphi^n)\big],  \qquad\text{where} \quad \s_S(\varphi):=\max_{s\in S}|\varphi(s)| , \\
    \overline{\mf{o}}_{\rm top}(\phi)&:=\big[\tau_S(\phi^n)\big],  \qquad\text{where} \quad \tau_S(\phi):=\min_{x\in G}\max_{s\in S}|x\varphi(s)x^{-1}| .
\end{align*}
Note that $\overline{\mc{O}}_{\rm top}(\varphi)$ and $\overline{\mf{o}}_{\rm top}(\phi)$ are independent of the choice of $S$, and we have $\overline{\mc{O}}_{\rm top}(\varphi)\preceq\overline{\mf{o}}_{\rm top}(\phi)$.

\begin{lem}\label{lem:o_bar_is_faster}
    We have $\mc{O}\preceq\overline{\mc{O}}_{\rm top}(\varphi)$ for all $\mc{O}\in\mc{G}(\varphi)$, and $\mf{o}\preceq\overline{\mf{o}}_{\rm top}(\phi)$ for all $\mf{o}\in\mf{g}(\phi)$. 
\end{lem}
\begin{proof}
    For all $g\in G$ and $n\in\N$, we have $|\varphi^n(g)|\leq |g|_S\cdot\s^S(\varphi^n)$ and $\|\phi^n(g)\|\leq \|g\|_S\cdot\tau^S(\phi^n)$. The lemma immediately follows from these two inequalities.
\end{proof}

As mentioned, many natural questions about growth rates of general automorphisms seem to be open. I list here some that I find interesting. There seems to be no reason to expect that the answer to any of these should be `yes'.

\begin{quest}\label{quest}
    Let $G$ be finitely presented. Consider $\varphi\in\Aut(G)$ and its outer class $\phi\in\Out(G)$.
    \begin{enumerate}
        \item Do the sets $\mc{G}(\varphi)$ and $\mf{g}(\phi)$ always have a $\preceq$--maximum?
        \item Do the growth rates $\overline{\mc{O}}_{\rm top}(\varphi)$ and $\overline{\mf{o}}_{\rm top}(\phi)$ always lie in $\mc{G}(\varphi)$ and $\mf{g}(\phi)$, respectively?
        \item Are the sets $\mc{G}(\varphi)$ and $\mf{g}(\phi)$ always finite?
        \item Do we have $\overline{\mc{O}}_{\rm top}(\varphi)\sim\overline{\mf{o}}_{\rm top}(\phi)$ whenever $\overline{\mf{o}}_{\rm top}(\phi)$ grows at least exponentially?
        \item Is the stretch factor ${\rm str}(\phi):=\sup_{g\in G}\,\limsup_n\,\|\phi^n(g)\|^{1/n}$ always an algebraic integer?
        \item For an element $g\in G$, do the limits $\lim_n\,|\varphi^n(g)|^{1/n}$ and $\lim_n\,\|\phi^n(g)\|^{1/n}$ always exist?
        \item If an element of $\mc{G}(\varphi)\cup\mf{g}(\phi)$ is faster than any polynomial, does it grow at least exponentially?
    \end{enumerate}
\end{quest}

To the best of my knowledge, the above questions are open also for \emph{finitely generated} groups, except for Item~(7) which was settled in \cite{Coulon}. The following is another simple construction of exotic examples that was communicated to me by Martin Bridson.

\begin{ex}\label{ex:exotic}
    There exist finitely presented groups $G$ with automorphisms $\varphi\in\Aut(G)$ and $\phi\in\Out(G)$ having all sorts of growth rates that are bounded above by a polynomial function, while not being exactly polynomial themselves. In particular, one can arrange these so that $\preceq$ is not a total order on $\mc{G}(\varphi)$ or $\mf{g}(\phi)$. We briefly explain the construction here.

    Let $\alpha\colon\Z\ra\N$ be a subadditive function such that $\alpha^{-1}(0)=\{0\}$ and $\alpha(n)=\alpha(-n)$ for all $n$, and such that the cardinality of the preimages $\alpha^{-1}([0,n])$ grows at most exponentially in $n$. Let $H$ be a finitely generated group with an element $h\in H$ such that the function $n\mapsto |h^n|_H$ is bi-Lipschitz equivalent to $\alpha$. Such pairs $(H,h)$ always exist by \cite{Olsh1} and, if $\alpha$ is a computable function, one can always take $H$ to be finitely presented \cite{Olsh2}.

    Now, consider the free group $F=\langle x_1,\dots,x_m\rangle$ for some $m\geq 2$. Let $\psi\in\Aut(F)$ be the automorphism that fixes $x_1$ and maps $x_k\mapsto x_kx_{k-1}$ for all $k\geq 2$. Observing that $\psi^n(x_k)=x_k\cdot x_{k-1}\psi(x_{k-1})\dots\psi^{n-1}(x_{k-1})$, one easily deduces that $|\psi^n(x_k)|_F\sim n^{k-1}$. 
    
    Choosing $(H,h)$ as above, we can form the amalgam
    \[ G:= F \underset{x_1=h}{\ast} H ,\]
    which is finitely generated (resp.\ presented) if $H$ is. The automorphism $\psi$ extends to an automorphism $\varphi:=\psi\ast\id_H\in\Aut(G)$, and we denote by $\phi\in\Out(G)$ its outer class.
    \begin{enumerate}
    \setlength\itemsep{.25em}
        \item For the element $x_2$, we always get
         \[ \|\phi^n(x_2)\|_G \sim |\varphi^n(x_2)|_G \sim |x_2x_1^n|_G\sim |h^n|_H\sim \alpha(n) .\]
        In particular, the growth rate $[\alpha(n)]$ lies in both $\mc{G}(\varphi)$ and $\mf{g}(\phi)$.
        \item To describe the growth rate of $x_k$ for $k\geq 3$, we need to restrict $\alpha$ a little more: suppose that the restriction $\alpha|_{\N}\colon\N\ra\N$ is weakly increasing. This implies that $\sum_{j=1}^n\alpha(j) \sim n\alpha(n)$. Indeed, monotonicity gives $\frac{n}{2}\alpha(\frac{n}{2})\leq \sum_{j=1}^n\alpha(j)\leq n\alpha(n)$, and subadditivity yields $\frac{n}{2}\alpha(\frac{n}{2})\geq\frac{n}{4}\alpha(n)$. Armed with this equivalence, we obtain for all $k\geq 3$:
        \[ \|\phi^n(x_k)\|_G \sim |\varphi^n(x_k)|_G \sim \sum_{j=1}^{n-1}|\varphi^j(x_{k-1})|_G \sim n^{k-2}\alpha(n) .\]
        Thus, for each $0\leq k\leq m-2$, the growth rate $[n^k\alpha(n)]$ lies in $\mc{G}(\varphi)$ and $\mf{g}(\phi)$.
        \item Finally, we consider a double version of the above construction. Let $\alpha,\beta\colon\Z\ra\N$ be functions as at the start of the example, and let $(H,h)$ and $(L,\ell)$ be corresponding group/element pairs. Let $F'=\langle x_1,x_2,y_1,y_2\rangle$ be the free group on four generators, and consider the automorphism $\psi'\in\Aut(F')$ that fixes $x_1$ and $y_1$, and maps $x_2\mapsto x_2x_1$ and $y_2\mapsto y_2y_1$. 
        
        Consider the double amalgam
        \[ G':= L \underset{\ell=y_1}{\ast} F' \underset{x_1=h}{\ast} H ,\]
        and the automorphism $\varphi'\in\Aut(G')$ given by $\varphi':=\id_L\ast\psi'\ast\id_H$. Let again $\phi'\in\Out(G')$ be its outer class. Exactly as above, both growth rates $[\alpha(n)]$ and $[\beta(n)]$ now lie in $\mc{G}(\varphi')$ and $\mf{g}(\phi')$. We can choose $\alpha$ and $\beta$ so that these two growth rates are $\preceq$--incomparable: for instance, we can choose $\alpha$ arbitrarily and then construct $\beta$ as follows. The function $\beta$ alternates between being constant on long segments of $\N$, and growing linearly with very small slopes on other long segments of $\N$.
        We leave to the reader the technical check that this can be arranged so that $\beta\not\preceq\alpha$ and $\alpha\not\preceq\beta$, while also ensuring that $\beta$ is computable, sublinear, and with an exponential bound on the size of initial-segment preimages.

        This gives sets $\mc{G}(\varphi')$ and $\mf{g}(\phi')$ on which $\preceq$ is not a total order. However, these sets still have $\preceq$--maxima given by the growth rate $[\alpha(n)+\beta(n)]$, which is realised on $x_1y_1\in G$.
    \end{enumerate}
\end{ex}

\subsection{Tameness and docility}

In view of the pathologies outlined in \Cref{ex:exotic} and \Cref{quest}, it is often useful to restrict to automorphisms whose growth rates are of the following kinds, as this covers all the most natural examples.

\begin{defn}\label{defn:tame}
    An abstract growth rate $[x_n]\in\mf{G}$ is: 
    \begin{enumerate}
        \item \emph{pure} if $[x_n]\sim [n^p\lambda^n]$ for some $\lambda>1$ and $p\in\N$;
        \item \emph{$(\lambda,p)$--tame}, for some $\lambda>1$ and $p\in\N$, if we have $[x_n]\sim[a_n\lambda^n]$ for a weakly increasing sequence $a_n$ with $[1]\preceq [a_n]\preceq [n^p]$.
    \end{enumerate}
\end{defn}

Pure growth rates are clearly tame, but the converse does not hold. Moreover, pure growth rates form a \emph{totally} ordered subset of $(\mf{G},\preceq)$, while tame ones do not. Also note that tame rates are at least exponential, as we ask that $\lambda>1$. Tame growth rates are particularly useful when studying automorphisms of right-angled Artin groups and special groups, as there tameness can sometimes be shown even when purity remains out of reach \cite{Fio11}. 

Recall that, for any $\varphi\in\Aut(G)$ with outer class $\phi\in\Out(G)$, we have $\overline{\mc{O}}_{\rm top}(\varphi)\succeq\overline{\mf{o}}_{\rm top}(\phi)$.

\begin{defn}\label{defn:tame_aut}
    Let $G$ be finitely generated. Consider $\varphi\in\Aut(G)$ with outer class $\phi\in\Out(G)$. 
    \begin{enumerate}
    \item The automorphism $\varphi$ is \emph{sound} if we have $\overline{\mc{O}}_{\rm top}(\varphi)\sim\overline{\mf{o}}_{\rm top}(\phi)$. The outer automorphism $\phi$ is \emph{sound} if all its representatives $\varphi'\in\Aut(G)$ are sound.
    \item The automorphism $\varphi$ is \emph{docile} if, at the same time, $\varphi$ is sound and $\overline{\mc{O}}_{\rm top}(\varphi)$ is tame. The outer automorphism $\phi$ is \emph{docile} if $\varphi$ is docile (this is independent of the choice of representative, by \Cref{rmk:tameness_independent_of_representative} below). We will also say that $\varphi$ and $\phi$ are $(\lambda,p)$--docile if we wish to specify the parameters for which the rates $\overline{\mc{O}}_{\rm top}(\varphi)$ and $\overline{\mf{o}}_{\rm top}(\phi)$ are $(\lambda,p)$--tame.
    \end{enumerate}
\end{defn}

Automorphisms are sound when, under their iterates, word length does not grow much faster than conjugacy length: applying powers of the automorphism does not yield elements most of whose word length is due to a failure to be ``cyclically reduced''. Exponentially-growing automorphisms of free and surface groups are all sound \cite{BH92,Levitt-GAFA,Thurston-BAMS}, while inner automorphisms clearly are not. 

\begin{rmk}\label{rmk:tame_implies_sum-stable}
    If $[x_n]\in\mf{G}$ is tame, we have $\big[\sum_{i\leq n}x_i\big]\sim[x_n]$. Indeed, since $[x_n]\sim[a_n\lambda^n]$ for a weakly increasing sequence $a_n$ by definition, we obtain $\sum_{i\leq n}a_i\lambda^i\leq a_n\sum_{i\leq n}\lambda^i\preceq a_n\lambda^n$.
\end{rmk}

The need for \Cref{rmk:tame_implies_sum-stable} is precisely what motivated us to require the sequence $a_n$ in \Cref{defn:tame} to be weakly increasing (and that $\lambda\neq 1$). We will use the previous remark in various forms (in \Cref{subsec:invariant_GOGs} and \Cref{subsec:free_products}), one of which is the following. Given $\varphi\in\Aut(G)$ and an element $g\in G$, we can define the elements $g_n:=g\varphi(g)\varphi^2(g)\dots\varphi^n(g)$. If the growth rate $\big[\,|\varphi^n(g)|\,\big]$ is tame, then $|g_n|\preceq |\varphi^n(g)|$. 
This leads to the next observation:

\begin{rmk}\label{rmk:tameness_independent_of_representative}
    If $\phi\in\Out(G)$ has a docile representative $\varphi\in\Aut(G)$, then \emph{all} its representatives are docile. Indeed, suppose that $\varphi$ is docile and let $\psi(x)=g\varphi(x)g^{-1}$ be another representative. One the one hand, we immediately have $\overline{\mc{O}}_{\rm top}(\psi)\succeq\overline{\mf{o}}_{\rm top}(\phi)\sim \overline{\mc{O}}_{\rm top}(\varphi)$. On the other, setting again $g_n:=g\varphi(g)\varphi^2(g)\dots\varphi^n(g)$, we have
    \[ |\psi^n(x)|=|g_{n-1}\varphi^n(x)g_{n-1}^{-1}|\leq 2|g_{n-1}| + |\varphi^n(x)| ,\]
    for all $x\in G$ and $n\in\N$. Thus, recalling \Cref{rmk:tame_implies_sum-stable} and the definition of $\overline{\mc{O}}_{\rm top}(\cdot)$, we also get $\overline{\mc{O}}_{\rm top}(\psi)\preceq\overline{\mc{O}}_{\rm top}(\varphi)$. In conclusion $\overline{\mc{O}}_{\rm top}(\psi)\sim\overline{\mc{O}}_{\rm top}(\varphi)$, which implies that $\psi$ is docile.
\end{rmk}

Finally, we say that an element $[x_n]\in\mf{G}$ is \emph{sub-polynomial} if $[x_n]\preceq[n^p]$ for some $p\in\N$. An automorphism $\varphi\in\Aut(G)$ is \emph{sub-polynomial} if $\overline{\mc{O}}_{\rm top}(\varphi)$ is sub-polynomial.
Note that, a priori, the fact that an outer class $\phi\in\Out(G)$ has sub-polynomial $\overline{\mf{o}}_{\rm top}(\phi)$ does not imply that any automorphism representing $\phi$ is sub-polynomial. However, we have the following analogue of \Cref{rmk:tameness_independent_of_representative}.

\begin{rmk}\label{rmk:sub-polynomial_independent_of_representative}
    If $\phi\in\Out(G)$ has a representative $\varphi\in\Aut(G)$ with $\overline{\mc{O}}_{\rm top}(\varphi)\preceq n^p$ for some $p\in\N$, then all representatives $\psi\in\Aut(G)$ of $\phi$ satisfy the weaker inequality $\overline{\mc{O}}_{\rm top}(\psi)\preceq n^{p+1}$. This is shown as in the previous remark: if $\psi(x)=g\varphi(x)g^{-1}$, we have $|\psi^n(x)|\leq 2|g_{n-1}|+|\varphi^n(x)|$ and:
    \[ |g_n|\leq\sum_{i\leq n}|\varphi^i(x)|\preceq \sum_{i\leq n} i^p\leq n^{p+1}. \]
\end{rmk}

\section{Direct products}\label{sect:direct_products}

\subsection{Abelian factors}\label{sub:abelian_factors}

Consider a product of the form $G=H\x A$, where $H$ is a finitely generated group with trivial centre, and $A\cong\Z^N$ for some $N\geq 1$. The automorphism group $\Aut(G)$ can be described as follows. Consider the set $\mc{M}(H,A)$ of formal matrices 
\begin{equation}\label{eq:matrix}
    \begin{pmatrix} \varphi & 0 \\ \alpha & \psi \end{pmatrix}
\end{equation} 
where $\varphi\in\Aut(H)$, $\psi\in\Aut(A)\cong{\rm GL}_N(\Z)$ and $\alpha\in \mathrm{H}^1(H,A)$; in other words, $\alpha$ is a homomorphism $H\ra A$. We can make $\mc{M}(H,A)$ into a group by endowing it with a natural product:
\[ \begin{pmatrix} \varphi_1 & 0 \\ \alpha_1 & \psi_1 \end{pmatrix}\cdot \begin{pmatrix} \varphi_2 & 0 \\ \alpha_2 & \psi_2 \end{pmatrix}= \begin{pmatrix} \varphi_1\varphi_2 & 0 \\ \alpha_1\varphi_2+\psi_1\alpha_2 & \psi_1\psi_2 \end{pmatrix}. \]
There is an action $\mc{M}(H,A)\acts G$ given by 
\[ \begin{pmatrix} \varphi & 0 \\ \alpha & \psi \end{pmatrix}\cdot (h,a)=(\varphi(h),\alpha(h)+\psi(a)), \]
which corresponds to a homomorphism $\iota\colon\mc{M}(H,A)\ra\Aut(G)$. 

\begin{lem}
    If $G=H\x A$ as above, the map $\iota\colon\mc{M}(H,A)\ra\Aut(G)$ is a group isomorphism.
\end{lem}
\begin{proof}
    Injectivity is clear, so we only need to show that $\iota$ is surjective. Since $H$ has trivial centre, $A$ is the centre of $G$ and it is preserved by all elements of $\Aut(G)$. Given $\chi\in\Aut(G)$, we can set $\psi:=\chi|_A\in\Aut(A)$. Denoting by $\pi_H,\pi_A$ the two factor projections of $G$, we also set $\alpha:=\pi_A\o\chi|_H\in \mathrm{H}^1(H,A)$ and $\varphi:=\pi_H\o\chi|_H$. For the moment $\varphi$ is just a homomorphism $H\ra H$, but we can certainly write $\chi$ as $\chi(h,a)=(\varphi(h),\alpha(h)+\psi(a))$. We have $\chi(G)\sq\varphi(H)\x A$ and $\chi$ is an isomorphism, so $\varphi\colon H\ra H$ must be surjective. Similarly we have $\chi(\ker\varphi)\sq A$ and $\chi(A)=A$, so injectivity of $\chi$ implies injectivity of $\varphi$. This shows that $\varphi\in\Aut(H)$, and thus $\chi$ is in the image of $\iota\colon\mc{M}(H,A)\ra\Aut(G)$. Since $\chi\in\Aut(G)$ was arbitrary, this completes the proof of the lemma.
\end{proof}

To simplify inline notation, from now on we will denote by $\mc{M}(\varphi,\psi,\alpha)$ the element of $\Aut(G)$ that is the image under $\iota$ of the matrix in \Cref{eq:matrix}. For each $n\geq 1$, we have $\chi^n=\mc{M}(\varphi^n,\psi^n,\alpha_n)$, where $\alpha_n=\sum_{j=1}^n\psi^{j-1}\alpha\varphi^{n-j}$. Thus, the growth rates of $\mc{M}(\varphi,\psi,\alpha)$ and its projection to $\Out(G)$ can be described fairly easily in terms of the growth rates of $\varphi$ and $\psi$.

For a finitely generated group $Q$, we denote by $Q_{\rm ab}$ the free part of the abelianisation of $Q$. Every automorphism $\chi\in\Aut(Q)$ naturally descends to an automorphism of $\chi_{\rm ab}\in\Aut(Q_{\rm ab})$, and each element $q\in Q$ equivariantly projects to an element $q_{\rm ab}\in Q_{\rm ab}$.

\begin{lem}\label{lem:abelian_factor}
    Let $G=H\x A$ be as above. For each automorphism $\chi=\mc{M}(\varphi,\psi,\alpha)\in\Aut(G)$ and each element $g=(h,a)\in G$, we have the following equalities in the set of abstract growth rates $\mf{G}$:
        \begin{align*}
    |\chi^n(g)|&\sim|\varphi^n(h)|+|\chi_{\rm ab}^n(g_{\rm ab})| , & \|\chi^n(g)\|&\sim\|\varphi^n(h)\|+|\chi_{\rm ab}^n(g_{\rm ab})| .
        \end{align*}
    Moreover, there exist an algebraic integer $\lambda\geq 1$ and $p\in\N$ such that $|\chi_{\rm ab}^n(g_{\rm ab})|\sim n^p\lambda^n$.
\end{lem}
\begin{proof}
    The element $\chi^n(g)$ has coordinates $\varphi^n(h)$ and $\alpha_n(g)+\psi^n(a)$ in $H$ and $A$ respectively, where $\alpha_n$ is the homomorphism described above. We have $G_{\rm ab}=H_{\rm ab}\oplus A$, and the homomorphisms $\alpha_n\colon H\ra A$ factor through $H_{\rm ab}$. Thus, the coordinate of $\chi_{\rm ab}^n(g_{\rm ab})$ along $A$ is identical to that of $\chi^n(g)$. We obtain the inequality
    \begin{align*}
        |\chi^n(g)|\sim |\varphi^n(h)|+|\alpha_n(g)+\psi^n(a)|\preceq |\varphi^n(h)|+|\chi_{\rm ab}^n(g_{\rm ab})|,
    \end{align*}
    as well as the analogous inequality for $\|\chi^n(g)\|$. The reverse inequalities are immediate from the fact that the projections $G\ra H$ and $G\ra G_{\rm ab}$ are Lipschitz.

    Finally, we have $|\chi_{\rm ab}(g_{\rm ab})|\sim n^p\lambda^n$ for an algebraic integer $\lambda\geq 1$ and some $p\in\N$ because $G_{\rm ab}$ is free abelian and we can invoke the classical \Cref{lem:classical_abelian} below. Algebraic integers are closed under taking complex conjugates, products, square roots and thus also under taking moduli.
\end{proof}

\begin{lem}\label{lem:classical_abelian}
    Let $k\geq 1$. For each automorphism $\varphi\in\Aut(\Z^k)={\rm GL}_k(\Z)$ and each element $a\in\Z^k\setminus\{0\}$, there exist $p\in\N$ and $\mu\in\C$ with $|\mu|\geq 1$ such that
    \begin{align*}
        |\varphi^n(a)|=\|\varphi^n(a)\|\sim n^p|\mu|^n.
    \end{align*}
    Moreover, $\varphi\in{\rm GL}_k(\Z)\leq{\rm GL}_k(\C)$ has a $(p+1)$--dimensional Jordan block with eigenvalue $\mu$.
\end{lem}

For automorphisms of free and surface groups, the maximal subgroups all of whose elements do not grow at top speed under some automorphism are particularly well-behaved. In particular, they are always finitely generated. The next example shows that, when we introduce an abelian direct factor, this property no longer holds.

\begin{ex}\label{ex:positive_ifg}
    Consider the product $G=F_k\x\Z^k$ for $k\geq 3$. There exists an automorphism $\chi\in\Aut(G)$ that has exactly three growth rates --- namely $[1]$, $[\lambda^n]$ and $[n\lambda^n]$ for some $\lambda>1$ --- where $[1]$ is realised at the identity, $[\lambda^n]$ is realised by nontrivial elements of a subgroup $N\x\Z^k\lhd G$, where $N\lhd F_n$ is infinitely generated, and $[n\lambda^n]$ is realised by all remaining elements of $G$.

    The construction of such automorphisms $\chi$ is fairly general. We start with a positive, fully irreducible
    automorphism\footnote{
    A simple such example in $F_3=\langle a,b,c\rangle$ is the automorphism $\varphi$ mapping $a\mapsto ab$, $b\mapsto bc$, $c\mapsto cab$. 
    } 
    $\varphi\in\Aut(F_k)$. Let $\varphi_{\rm ab}$ be the induced automorphism of the abelianisation $(F_k)_{\rm ab}\cong\Z^k$ and let $\alpha\colon F_k\ra (F_k)_{\rm ab}$ be the quotient projection. Set $G:=F_k\x (F_k)_{\rm ab}\cong F_k\x\Z^k$ and $\chi:=\mc{M}(\varphi,\varphi_{\rm ab},\alpha)$, in the above notation. We now prove the above statements.

    Let $\lambda$ be the Perron--Frobenius eigenvalue of $\varphi$. By the existence of train-track maps, we have $|\varphi^n(h)|\sim\|\varphi^n(h)\|\sim\lambda^n$ for all nontrivial elements $h\in F_k$ \cite{BH92}. We also have $|\varphi_{\rm ab}^n(\alpha(h))|\leq |\varphi^n(h)|$ for all $h\in F_k$, in terms of the standard generating sets of $F_k$ and $(F_k)_{\rm ab}$. This inequality is an equality whenever $h$ is a positive element of $F_k$, 
    since $\varphi$ is a positive automorphism.
    
    Note that $\chi^n=\mc{M}(\varphi^n,\varphi_{\rm ab}^n,\alpha_n)$ for each $n\geq 1$, where $\alpha_n=n\cdot (\alpha\o\varphi^{n-1})=n\cdot (\varphi_{\rm ab}^{n-1}\o\alpha)$. Thus,
    \[ \chi^n(h,a)=\big(\varphi^n(h),\ \varphi_{\rm ab}^n(a)+n\cdot(\varphi_{\rm ab}^{n-1}\alpha)(h)\big) .\]
    It follows that $\lambda^n\preceq|\chi^n(g)| \preceq n\lambda^n$ for every nontrivial element $g\in G$. In addition, we have $|\chi^n(g)|\sim n\lambda^n$ if $g=(h,\alpha(h))\in F_k\x(F_k)_{\rm ab}$ for a positive element $h\in F_k$. The same statements hold for conjugacy-length growth.

    Now, let $\beta\colon F_k\ra\R^k$ be the limit for $n\ra+\infty$ of the homomorphisms 
    \[ \frac{1}{n\lambda^n}\alpha_n=\lambda^{-n}(\varphi_{\rm ab}^n\o\alpha)\colon F_k\ra(F_k)_{\rm ab}\cong\Z^k\hookrightarrow\R^k .\]
    That this limit exists follows from the fact that a power of $\varphi_{\rm ab}\in{\rm GL}_k(\Z)$ is represented by a matrix with positive entries and the Perron--Frobenius eigenvalue of $\varphi_{\rm ab}$ is $\lambda$, so all other eigenvalues of $\varphi_{\rm ab}$ have modulus $<\lambda$. 
    The homomorphism $\beta$ is nontrivial, as it does not vanish on positive elements of $F_k$, by the above discussion.
    
    Any nontrivial element $g$ of $\ker\beta\x\Z^k$ clearly has $|\chi^n(g)|\sim \|\chi^n(g)\|\sim\lambda^n$, 
    while all elements of $G$ outside $\ker\beta\x\Z^k$ have $[n\lambda^n]$ as their growth rate, by construction.
    This proves all our claims.
\end{ex}

\subsection{The general case}\label{subsec:direct_products}

Consider a finitely generated group $G=G_1\x\dots\x G_k\x A$, where $k\geq 0$ and $A$ is free abelian. Assume that each $G_i$ is directly indecomposable and has trivial centre.

\begin{lem}\label{lem:aut_preserve_factors}
    The automorphism group $\Aut(G)$ permutes the subgroups $\langle G_i,A\rangle$ for $1\leq i\leq k$. 
\end{lem}
\begin{proof}
    Consider some $\varphi\in\Aut(G)$. Since $A$ is the centre of $G$, we have $\varphi(A)=A$. Thus, $G$ is generated by the subgroups $\varphi(G_1),\dots,\varphi(G_k),A$. Denoting by $\pi_1\colon G\ra G_1$ the factor projection, it follows that $G_1$ is generated by the pairwise-commuting subgroups $\pi_1\varphi(G_1),\dots,\pi_1\varphi(G_k)$. Since $G_1$ has trivial centre, each subgroup $\pi_1\varphi(G_i)$ has trivial intersection with the subgroup generated by the other $k-1$ subgroups $\pi_1\varphi(G_j)$. It follows that 
    \[ G_1=\pi_1\varphi(G_1)\x\dots\x\pi_1\varphi(G_k).\]
    Since $G_1$ is directly indecomposable, we must have $G_1=\pi_1\varphi(G_i)$ for some index $i$ and $\pi_1\varphi(G_j)=\{1\}$ for all $j\neq i$. Repeating the argument for all factors of $G$, we obtain a permutation $\s\in{\rm Sym}(k)$ such that $G_i=\pi_i\varphi(G_{\s(i)})$ for all $1\leq i\leq k$ and $\pi_i\varphi(G_j)=\{1\}$ for all $j\neq\s(i)$. This shows that $\langle\varphi(G_j),A\rangle = \langle G_{\s^{-1}(j)},A\rangle$ for all $1\leq j\leq k$, completing the proof.
\end{proof}

Let $\Aut^0(G)\leq\Aut(G)$ be the finite-index subgroup preserving each subgroup $\langle G_i,A\rangle$. When $A=\{1\}$, the elements of $\Aut^0(G)$ are simply products $\varphi_1\x\dots\x\varphi_k$ with $\varphi_i\in\Aut(G_i)$. We can now apply \Cref{lem:abelian_factor} to immediately deduce the following.

\begin{cor}\label{cor:growth_direct_product}
    Let $G=G_1\x\dots\x G_k\x A$ be as above. For each $\varphi\in\Aut^0(G)$, there exist automorphisms $\varphi_i\in\Aut(G_i)$ such that, for each element $g=(g_1,\dots,g_k,a)\in G$, we have:
    \begin{align*}
            |\varphi^n(g)|&\sim\sum_{i=1}^k|\varphi_i^n(g_i)|+|\varphi_{\rm ab}^n(g_{\rm ab})| , & \|\varphi^n(g)\|&\sim\sum_{i=1}^k\|\varphi_i^n(g_i)\|+|\varphi_{\rm ab}^n(g_{\rm ab})| .
    \end{align*}
    Moreover, for each $\varphi$ and $g$, there exist an algebraic integer $\lambda\geq 1$ and $p\in\N$ with $|\varphi_{\rm ab}^n(g_{\rm ab})|\sim n^p\lambda^n$.
\end{cor}

\section{Invariant graphs of groups}\label{subsec:invariant_GOGs}

Consider a finitely generated group $G$ and an outer automorphism $\phi\in\Out(G)$ that preserves a splitting of $G$ as a graph of groups. The next lemma is a classical exercise in Bass--Serre theory.

\begin{lem}\label{lem:automorphisms_fixing_BS}
    Consider a group $G$ with a one-edge splitting as
    \[ G=A\ast_C B \qquad\text{ or }\qquad G=A\ast_{C,\g}=\langle A,t\mid t^{-1}ct=\g(c),\ \forall c\in C\rangle .\]
    If, for some $\phi\in\Out(G)$, the Bass--Serre tree $G\acts T$ extends to an action $G\rtimes_{\phi}\Z\acts T$ without inversions, then $\phi$ is represented by an automorphism $\varphi\in\Aut(G)$ of the following form:
    \begin{enumerate}
        \item In the amalgamated product case, we have $\varphi(A)=A$ and $\varphi(B)=B$.
        \item In the HNN case, we have $\varphi(A)=A$, $\varphi(C)=C$ and $\varphi(t)=ta$ for some $a\in A$. 
    \end{enumerate}
\end{lem}

The next proposition is the main result of this section. The main takeaway is that, in a $\phi$--invariant graph of groups, growth on the whole group cannot be strictly faster than growth on the vertex groups, provided that the latter is sufficiently well-behaved and at least exponential. This can often be used to describe the top growth rate of $\phi$ on the whole group. At the same time, understanding the behaviour of \emph{all} growth rates of $\phi$ is a much more delicate problem, which we do not address in this short article.

A \emph{splitting} of a group $G$ is a minimal action on a simplicial tree $G\acts T$ that has at least one edge, and no edge-inversions. The splitting is \emph{$\phi$--invariant} for some $\phi\in\Out(G)$ if the $G$--action on $T$ extends to a $G\rtimes_{\phi}\Z$--action on $T$. 
Recall that the growth rates $\overline{\mc{O}}_{\rm top}(\cdot)$ and $\overline{\mf{o}}_{\rm top}(\cdot)$ were defined in \Cref{sub:general_growth}, and docile automorphisms were introduced in \Cref{defn:tame_aut}. Given $\phi\in\Out(G)$ and a subgroup $H\leq G$ with $\phi$--invariant $G$--conjugacy class, the restrictions $\phi|_H\in\Out(H)$ were discussed in \Cref{rmk:restriction_new}.

\begin{prop}\label{prop:growth_GOGs}
    Consider a finitely generated group $G$ with a $\phi$--invariant splitting $G\acts T$ for some $\phi\in\Out(G)$.
    Suppose that $\phi$ descends to the identity on the finite graph $T/G$. Additionally, suppose that there exist integers $p,q\in\N$ and a finite subset $\Lambda\sq\R_{>1}$ such that, for each vertex $v\in T$, the $G$--stabiliser of $v$ (denoted $V$) satisfies {\bf at least one} of the following conditions:
    \begin{enumerate}
        \item[(a)] for every representative $\varphi\in\Aut(G)$ of $\phi$ with $\varphi(V)=V$ and every $g\in V$, we have $|\varphi^n(g)|\preceq n^q$ (computing word lengths with respect to a finite generating set of $G$);
        \item[(b)] $V$ is finitely generated, conjugacy-undistorted in $G$ and, for every representative $\varphi\in\Aut(G)$ of $\phi$ with $\varphi(V)=V$, the restriction $\varphi|_V\in\Aut(V)$ is $(\lambda,p)$--docile for some $\lambda\in\Lambda$.
    \end{enumerate}
    Then, the following properties hold.
    \begin{enumerate}
        \item If there are no type~(b) vertex groups, then we have $\overline{\mc{O}}_{\rm top}(G,\varphi)\preceq n^{q+2}$ for all representatives $\varphi\in\Aut(G)$ of $\phi$.
        \item If there is at least one type~(b) vertex group, then $\phi\in\Out(G)$ is $(\mu,p)$--docile for $\mu=\max\Lambda$. Moreover, we can choose finitely many type~(b) vertex groups $V_1,\dots,V_k\leq G$ and restrictions $\phi_i\in\Out(V_i)$ of the outer automorphism $\phi$ such that $\overline{\mf{o}}_{\rm top}(\phi)\sim\sum_{i=1}^k\overline{\mf{o}}_{\rm top}(\phi_i)$.
    \end{enumerate}
\end{prop}
\begin{proof}
    To begin with, we prove the proposition when $T$ is a one-edge splitting of $G$, that is, an amalgamation or an HNN extension. In each of these two cases, we represent $\phi$ by an automorphism $\varphi\in\Aut(G)$ of the form in \Cref{lem:automorphisms_fixing_BS}. At the end of the proof, we will briefly explain how to handle general graphs of groups based on this.

    Suppose first that $G=A\ast_C B$ with $\varphi(A)=A$ and $\varphi(B)=B$. A general element of $G$ can be written as $g=a_1b_1\dots a_kb_k$ with $a_i\in A$ and $b_i\in B$, and we have
    \[\varphi^n(g)=\varphi^n(a_1)\varphi^n(b_1)\dots\varphi^n(a_k)\varphi^n(b_k) , \qquad\forall n\in\N .\] 
    This yields the simple bound $|\varphi^n(g)|\preceq\sum_{i=1}^k|\varphi^n(a_i)|+\sum_{i=1}^k|\varphi^n(b_i)|$. In part~(1), we immediately obtain $|\varphi^n(g)|\preceq n^q$ for all $g\in G$, and hence $\overline{\mc{O}}_{\rm top}(\varphi)\preceq n^q$. Regarding part~(2), suppose for simplicity that $A$ is of type~(a) and $B$ is of type~(b) (the case when both $A$ and $B$ are of type~(b) is similar). As $\varphi|_B$ is docile, every element $g\in G$ satisfies
    \[ |\varphi^n(g)|\preceq n^q+\overline{\mc{O}}_{\rm top}(\varphi|_B)\preceq \overline{\mc{O}}_{\rm top}(\varphi|_B)\sim\overline{\mf{o}}_{\rm top}([\varphi|_B]) . \] 
    Hence $\overline{\mc{O}}_{\rm top}(\varphi)\preceq \overline{\mf{o}}_{\rm top}([\varphi|_B])$. At the same time, the fact that $B$ is conjugacy-undistorted implies that $\overline{\mf{o}}_{\rm top}(\phi)\succeq \overline{\mf{o}}_{\rm top}([\varphi|_B])$. Thus, we obtain $
    \overline{\mc{O}}_{\rm top}(\varphi)\sim\overline{\mf{o}}_{\rm top}(\phi)\sim\overline{\mf{o}}_{\rm top}([\varphi|_B])$, showing that $\phi$ is docile, as required. Finally, if both vertex groups $A$ and $B$ are of type~(b), the same arguments show that $
    \overline{\mc{O}}_{\rm top}(\varphi)\sim\overline{\mf{o}}_{\rm top}(\phi)\sim\overline{\mf{o}}_{\rm top}([\varphi|_A])+\overline{\mf{o}}_{\rm top}([\varphi|_B])$.

    Suppose now that $G=A\ast_C$ with $\varphi(A)=A$, $\varphi(C)=C$ and $\varphi(t)=ta$, where $t$ is the stable letter of the HNN extension and $a\in A$ is some element. A general element of $G$ can be written as $g=x_0t^{\eps_1}x_1\dots t^{\eps_k}x_k$ with $x_i\in A$ and $\eps_i\in\{\pm 1\}$. We then have 
    \[ \varphi^n(g)=x_0't^{\eps_1}x_1'\dots t^{\eps_k}x_k' ,\] 
    where, setting $a_n:=a\varphi(a)\varphi^2(a)\dots\varphi^{n-1}(a)$, each element $x_i'$ is one of the following four options (depending on the values of $\eps_i$ and $\eps_{i+1}$): either $\varphi^n(x_i)$, or $a_n\varphi^n(x_i)$, or $\varphi^n(x_i)a_n^{-1}$, or $a_n\varphi^n(x_i)a_n^{-1}$. This yields the inequality
    \[ \overline{\mc{O}}_{\rm top}(\varphi)\preceq \overline{\mc{O}}_{\rm top}(\varphi|_A) + |a_n| . \]
    In part~(1), we have $|a_n|\preceq n^{q+1}$ and we obtain $\overline{\mc{O}}_{\rm top}(\varphi)\preceq n^{q+1}$. In part~(2), the fact that $\varphi|_A$ is docile yields $|a_n|\preceq \overline{\mc{O}}_{\rm top}(\varphi|_A)$ (\Cref{rmk:tame_implies_sum-stable}), and hence $\overline{\mc{O}}_{\rm top}(\varphi)\preceq \overline{\mc{O}}_{\rm top}(\varphi|_A)$. At the same time, since $A$ is conjugacy-undistorted, we have $\overline{\mf{o}}_{\rm top}(\phi)\succeq\overline{\mf{o}}_{\rm top}([\varphi|_A])$ and as above $
    \overline{\mc{O}}_{\rm top}(\varphi)\sim\overline{\mf{o}}_{\rm top}(\phi)\sim\overline{\mf{o}}_{\rm top}([\varphi|_A])$, showing that $\phi$ is docile.

    So far, we have only considered a \emph{specific} representative $\varphi$ of the outer class $\phi$. This is irrelevant in part~(2), while in part~(1) it causes the bound $\overline{\mc{O}}_{\rm top}(\varphi)\preceq n^{q+1}$ to translate itself into the weaker bound $\overline{\mc{O}}_{\rm top}(\varphi')\preceq n^{q+2}$ for a general representative $\varphi'$ (\Cref{rmk:sub-polynomial_independent_of_representative}).
    
    This proves the proposition when $G\acts T$ is a one-edge splitting. In general, we can decompose the $\phi$--invariant splitting $G\acts T$ as a finite sequence of $\phi$--invariant one-edge splittings (collapsing all $G$--orbits of edges of $T$ but one, and then adding them back one at a time).
    An iterated application of the one-edge case then immediately proves part~(2). Part~(1) requires a little more care, as we want the polynomial exponent $q$ to increase by at most $2$. For this, it suffices to note that $T$ can be decomposed as a finite sequence of amalgamated products followed by a single multiple HNN extension (a splitting whose quotient graph is a wedge of circles, with a single vertex). Amalgamations do not increase $q$ at all (for a specific representative $\varphi$), while the final multiple HNN splitting increases it by at most $1$ (here the argument is identical to the single HNN case, except that we will have finitely many stable letters $t_i$ and, in general, $\varphi$ can only be put in the form $\varphi(t_i)=a_ita_i'$ with $a_i,a_i'\in A$).
    Finally, there is a further increase by $1$ to handle arbitrary representatives $\varphi'$ using \Cref{rmk:sub-polynomial_independent_of_representative}. This concludes the proof of the proposition.
\end{proof}

\begin{rmk}
    If in part~(2) of \Cref{prop:growth_GOGs} the growth rates $\overline{\mf{o}}_{\rm top}(\phi_i)$ are all pure (\Cref{defn:tame}), then $\overline{\mf{o}}_{\rm top}(\phi)\sim\sum_i\overline{\mf{o}}_{\rm top}(\phi_i)$ simply equals the fastest of the $\overline{\mf{o}}_{\rm top}(\phi_i)$.
\end{rmk}

\section{Free products}\label{subsec:free_products}

A \emph{free splitting} is a minimal action on a simplicial tree $G\acts T$ with trivial edge-stabilisers, with at least one edge, and without edge-inversions. We say that the splitting is \emph{relative} to a family of subgroups $\mc{H}$ if each of the subgroups in $\mc{H}$ fixes a vertex of $T$. Free splittings are typically \emph{not} invariant under the automorphisms that we wish to study (so \Cref{subsec:invariant_GOGs} does not apply), though we can usually place ourselves in the situation where at least the collection of elliptic subgroups of the splitting is preserved by the automorphism.

Let $G$ be a finitely generated group. A \emph{(free) factor system} for $G$ is the collection $\mc{F}$ of all $G$--conjugates of the subgroups $G_1,\dots,G_k\leq G$ appearing in a decomposition 
\[ G=G_1\ast\dots\ast G_k\ast F_m \] 
with $k,m\geq 0$ and $k+m\geq 2$. We require the $G_i$ to be nontrivial, but not that they be freely indecomposable; we allow $\mc{F}$ to be empty if $G\cong F_m$ for $m\geq 2$. The group $G$ admits a factor system whenever it is neither freely indecomposable nor isomorphic to $\Z$. 

Following \cite{GH22}, we say that the pair $(G,\mc{F})$ is \emph{sporadic} if $(k,m)\in\{(2,0),(1,1)\}$. A \emph{$(G,\mc{F})$--free factor} is a subgroup of $G$ arising as a vertex group in a free splitting of $G$ relative to $\mc{F}$. A $(G,\mc{F})$--free factor is \emph{proper} if it is neither the trivial group nor an element of $\mc{F}$. Let $\Out(G,\mc{F})\leq\Out(G)$ be the subgroup of outer automorphisms that leave invariant each $G$--conjugacy class of subgroups in $\mc{F}$. An element $\phi\in\Out(G,\mc{F})$ is \emph{fully irreducible} if none of its (nontrivial) powers preserves the $G$--conjugacy class of a proper $(G,\mc{F})$--free factor. 

The following is an equivalent characterisation of full irreducibles. It shows that, up to taking powers and changing the factor system, we can always reduce to studying full irreducibles. If $\mc{F}_1,\mc{F}_2$ are factor systems for $G$, we write $\mc{F}_1\leq\mc{F}_2$ if each subgroup in $\mc{F}_1$ is contained in a subgroup in $\mc{F}_2$. 

\begin{lem}\label{lem:enlarging_factor_system}
   If $\phi\in\Out(G,\mc{F})$ is not fully irreducible, then there exist a factor system $\mc{F}'>\mc{F}$ and an integer $p\geq 1$ such that $\phi^p\in\Out(G,\mc{F}')$.
\end{lem}
\begin{proof}
    Write $G=G_1\ast\dots\ast G_k\ast F_m$, where the elements of $\mc{F}$ are precisely the conjugates of the $G_i$. If $\phi\in\Out(G,\mc{F})$ is not fully irreducible, let $H$ be a proper $(G,\mc{F})$--free factor preserved by a power $\phi^p$ with $p\geq 1$. Define 
    \[ \mc{F}':=\{gHg^{-1}\mid g\in G\}\cup\mc{F}_0 ,\] 
    where $\mc{F}_0\sq\mc{F}$ is the subset of subgroups that are not contained in any $G$--conjugate of $H$.

    By definition, $G$ has a free splitting relative to $\mc{F}$ in which $H$ is a vertex group; let $V_1,\dots,V_s$ be representatives of the other $G$--conjugacy classes of (nontrivial) vertex groups in this splitting. Thus, we have $G=H\ast V_1\ast\dots V_s\ast F_{m'}$ for some $m'\geq 0$. Each $V_i$ is a free product of a free group and $G$--conjugates of some of the $G_j$ (which happen to lie in $V_i$). Thus, we can also write $G=H\ast G_1'\ast\dots\ast G_{k'}'\ast F_{m''}$, where each $G_i'$ is a $G$--conjugate of some $G_j$ and the $G$--conjugates of the $G_i'$ are precisely the elements of $\mc{F}_0$. This shows that $\mc{F}'$ is indeed a factor system for $G$.

    It is immediate that $\phi^p\in\Out(G,\mc{F}')$ and that $\mc{F}'>\mc{F}$, so this completes the proof.
\end{proof}

As a first step, we can characterise growth of fully irreducible automorphisms in terms of their growth on the elements of the factor system. This is a relatively straightforward (but also rather fiddly) application of train track maps for free products \cite{CT94,FM15,Lyman-train-track}. 

Recall that $\mc{G}(\varphi)\sq\mf{G}$ and $\mf{g}(\phi)\sq\mf{G}$ denote the sets of growth rates of an automorphism $\varphi$ and an outer automorphism $\phi$, respectively. The rate $\overline{\mf{o}}_{\rm top}(\phi)$ always bounds all elements of $\mf{g}(\phi)$ from above, but it might not lie in $\mf{g}(\phi)$ a priori; the same is true of $\overline{\mc{O}}_{\rm top}(\varphi)$ and $\mc{G}(\varphi)$. Also recall that we refer to polynomial-times-exponential growth rates as \emph{pure} (\Cref{defn:tame}).

\begin{prop}\label{prop:growth_estimate_fully_irreducible}
    Consider a finitely generated group $G=G_1\ast\dots\ast G_k\ast F_m$, let $\mc{F}$ be the factor system given by the $G_i$, and let $\phi\in\Out(G,\mc{F})$. Suppose that $\phi$ is fully irreducible and $\mc{F}$ is non-sporadic. In addition, suppose that each restriction $\phi_i:=\phi|_{G_i}\in\Out(G_i)$ is represented by some $\varphi_i\in\Aut(G_i)$ such that either $\overline{\mc{O}}_{\rm top}(\varphi_i)$ is sub-polynomial, or $\varphi_i$ is $(\lambda_i,p_i)$--docile for some $\lambda_i>1$ and $p_i\in\N$. Then there exists a Perron number $\lambda>1$ such that all the following hold.
    \begin{enumerate}
        \item The automorphism $\phi$ is $(\mu,q)$--docile for $\mu=\max\{\lambda_1,\dots,\lambda_k,\lambda\}$ and some $q\in\N$. We have $q\leq p_j+1$ for the largest integer $p_j$ such that $\lambda_j=\mu$; if no such index $j$ exists, then $q=0$.
        \item If we have $[\lambda_i^n]\in\mf{g}(\phi_i)$ for all indices $i$ such that $\lambda_i=\mu$, then $[\mu^n]\in\mf{g}(\phi)$.
        \item Suppose that there exist a finite set $\Lambda\sq\R_{>1}$ and $P\in\N$ such that, for all indices $i$ and all representatives $\varphi_i'\in\Aut(G_i)$ of $\phi_i\in\Out(G_i)$, each growth rate in the union $\mc{G}(\varphi_i')\cup\mf{g}(\phi_i)$ is either sub-polynomial or equal to $[n^a\nu^n]$, for some $\nu\in\Lambda$ and some integer $0\leq a\leq P$. Then, each growth rate in the set $\mf{g}(\phi)$ is either sub-polynomial, or equal $[\lambda^n]$, or equal to $[n^a\nu^n]$ for some $\nu\in\Lambda$ and some integer $0\leq a\leq P+1$.
    \end{enumerate}
\end{prop}

\begin{proof}
    We begin with a general discussion that only assumes finite generation of $G$ and no additional properties of the restrictions $\phi|_{G_i}$; this discussion will culminate in Equations~\ref{eq:total_length_estimate} and~\ref{eq:O_estimate} below, estimating the word-length growth of elements not conjugate into an element of $\mc{F}$, and the top growth rate of $\phi$, respectively. After this, we will add the docility assumption and draw the necessary conclusions. Throughout, we omit some technical details in the interest of overall clarity.

    We start by representing $G$ as the fundamental group of a graph of groups $\mc{G}$ with trivial edge groups and precisely the elements of $\mc{F}$ as $G$--conjugates of the nontrivial vertex groups. We can find $\mc{G}$ so that, in addition, $\phi$ is realised by a relative train track map $f\colon\mc{G}\ra\mc{G}$; see e.g.\ \cite[Section~1]{Lyman-CTs}. Let $\overline{\mc{G}}$ be the finite graph underlying the graph of groups $\mc{G}$ and let $\overline f\colon\overline{\mc{G}}\ra\overline{\mc{G}}$ be the map of graphs underlying $f$. Choosing a base vertex $p$ and a spanning tree for $\overline{\mc{G}}$, we identify the vertex groups of $\mc{G}$ with the $G_i$, and we fix a representative $\varphi\in\Aut(G)$ of the outer class $\phi$. For notational convenience, choose automorphisms $\varphi_i\in\Aut(G_i)$ representing the restrictions $\phi_i\in\Out(G_i)$, and denote by $\psi$ the self-bijection of the disjoint union $\bigsqcup_i G_i$ that equals $\varphi_i$ on each $G_i$ (note that $\psi$ is not the restriction of a single element of $\Aut(G)$ in general).
    
    A \emph{generalised path} is a string $\pi=g_0e_1g_1\dots e_sg_s$ such that $e_1\dots e_s$ is an (oriented) edge path in $\overline{\mc{G}}$ and each $g_j$ lies in the vertex group $G_i$ associated to the terminal vertex of $e_j$ (which is also the initial vertex of $e_{j+1}$). The \emph{edge length} of $\pi$ is $\ell(\pi):=s$ and the \emph{total length} is $|\pi|:=s+\sum_j|g_j|$, where the word lengths $|g_j|$ are computed with respect to some fixed choice of finite generating sets for the $G_i$. The generalised path $\pi$ is \emph{immersed} if we do not have $g_j=1$ and $e_{j+1}=e_j^{-1}$ for any index $j$. Elements of $G$ are in $1$--to--$1$ correspondence with closed, immersed generalised paths based at $p\in\overline{\mc{G}}$. The word length of an element $g\in G$ is roughly the same as the total length of the generalised path $\pi$ representing $g$, up to a multiplicative constant independent of $g$.

    Since $\phi$ is fully irreducible, no power of $\overline f$ leaves invariant a proper subgraph of $\overline{\mc{G}}$ with at least one edge (up to collapsing some edges of $\mc{G}$ without altering the vertex groups). In particular, the train track map $f$ has only one stratum containing edges; let $\lambda\geq 1$ be the Perron--Frobenius eigenvalue of its transition matrix. Up to raising $\phi$ and $f$ to a power, we can further assume that:
    \begin{enumerate}
        \item[(i)] $\overline f$ fixes every vertex of $\overline{\mc{G}}$;
        \item[(ii)] for each edge $e\sq\mc{G}$, the path $\overline f(e)$ contains all edges of $\overline{\mc{G}}$;
        \item[(iii)] for each (oriented) edge $e\sq\mc{G}$, the edge paths $\overline {f^n}(e)\sq\overline{\mc{G}}$ all start with the same edge $e_1$ and end with the same edge $e_2$, for $n\geq 1$.
    \end{enumerate}
    Since $\mc{F}$ is non-sporadic, we do not have $\overline f(e)=e$ for any edge $e$. Thus, the edge lengths $\ell(f^n(e))$ grow exponentially with $n$ and hence $\lambda\neq 1$.

    If $\pi=g_0e_1g_1\dots e_sg_s$ is a generalised path, its image $f(\pi)$ is $\psi(g_0)f(e_1)\psi(g_1)\dots f(e_s)\psi(g_s)$, recalling that we have defined $\psi=\varphi_i$ on each $G_i$. Here, each $f(e_i)$ is a generalised path having $\overline f(e_i)$ as underlying edge path. 
    
    In order to easily estimate the total length of the paths $f^n(e)$, we define one last auxiliary concept. If $\pi=g_0e_1g_1\dots e_sg_s$, we refer to the slots between consecutive edges as the \emph{nodes} of $\pi$; thus, $\pi$ has $s+1$ nodes and they are occupied by the elements $g_0,\dots,g_s$. To each node of $f^n(\pi)$, we inductively associate an \emph{order}, which is an integer between $0$ and $n$. All nodes of $\pi$ have order $0$. If $e$ is an edge of $f^{n-1}(\pi)$, then $f(e)$ is a subpath of $f^n(\pi)$; all nodes of $f^n(\pi)$ that are interior nodes of $f(e)$ for some edge $e\sq f^{n-1}(\pi)$ are also declared to have order $0$. Every other node of $f^n(\pi)$ is the image under $f$ of a node of $f^{n-1}(\pi)$; if the node of $f^{n-1}(\pi)$ had order $m$, for some $m\geq 0$, then we declare the corresponding node of $f^n(\pi)$ to have order $m+1$. The initial and terminal nodes of $f^n(\pi)$ thus have order $n$.

    A straightforward computation shows that, for all $0\leq i<n$, the path $f^n(\pi)$ has exactly $\ell(f^{n-i}(\pi))-\ell(f^{n-i-1}(\pi))$ order--$i$ nodes, and it has $\ell(\pi)+1$ order--$n$ nodes. Choose a constant $C\geq 1$ such that every edge $e\sq\mc{G}$ satisfies
    \[ \tfrac{1}{C}\lambda^n \leq \ell(f^n(e))\leq C\lambda^n, \qquad \forall n\in\N . \]
    Since $\ell(f(e))\geq 2$ for all edges $e\sq\mc{G}$, we also have $\ell(f^n(e))-\ell(f^{n-1}(e))\geq\ell(f^{n-1}(e))\geq \frac{1}{\lambda C}\lambda^n$.
    
    Let $\Om$ be the finite set of group elements appearing at the nodes of the paths $f(e)$, as $e$ varies through the edges of $\mc{G}$; we add the identity $1\in G$ to $\Om$ to simplify notation in the coming discussion. For any $n\geq 0$, every order--$0$ node of $f^n(e)$ is occupied by an element lying in $\Om$. Property~(iii) implies that, for $1\leq i\leq n$, every order--$i$ node of $f^n(e)$ is occupied by an element of the form 
    \[ \left(a\psi(a)\psi^2(a)\dots\psi^{i-2}(a)\right)\cdot\psi^{i-1}(b)\psi^i(c)\psi^{i-1}(d)\cdot\left( \psi^{i-2}(e)\dots\psi^2(e)\psi(e)e\right), \quad \text{where } a,b,c,d,e\in\Om .\]
    In particular, fixing $i$, there is only a fixed finite set of elements that can occupy the order--$i$ nodes in the paths $f^n(e)$, as $n$ and $e$ are allowed to vary. Moreover, this finite set has at most $|\Om|^5$ elements, independently of $i$.
    Denote by $N_i$ the average word length of these finitely many elements.

    For any edge $e\sq\mc{G}$, we can write 
    \[  |f^n(e)| = \ell(f^n(e)) + \sum_{i=0}^{n-1} N_i^{e,n}\cdot \big(\ell(f^{n-i}(e))-\ell(f^{n-i-1}(e)) \big) + N_n^{e,n}(\ell(e)+1) ,\]
    where $N_i^{e,n}$ is the (weighted) average word length of the elements occupying the order--$i$ nodes of $f^n(e)$. Recalling property~(ii), the $N_i^{e,n}$ are roughly equal to $N_i$, up to a multiplicative constant independent of $e,i,n$. Hence, up to slightly enlarging the constant $C$ chosen above, we have
    \[ \frac{1}{C}\cdot \sum_{i=0}^n N_i\lambda^{n-i} \leq |f^n(e)| \leq  C \sum_{i=0}^n N_i\lambda^{n-i} , \qquad \text{for all $n\in\N$ and $e\sq\mc{G}$.} \]
    For any immersed path $\pi=g_0e_1g_1\dots e_sg_s$, denote by $|f^n(\pi)|_{\rm pt}$ the length of the path $f^n(\pi)$ \emph{pulled tight}, that is, the length of the immersed path homotopic to $f^n(\pi)$. If the length $|f^n(\pi)|_{\rm pt}$ does not stay bounded as $n$ increases, the above inequalities allow us to estimate
    \begin{equation}\label{eq:total_length_estimate}
        |f^n(\pi)|_{\rm pt}\sim \sum_{j=0}^s|\psi^n(g_j)| + \ell(\pi)\cdot \sum_{j=0}^n N_j\lambda^{n-j} ,
    \end{equation}
    where the multiplicative constant implicit in the symbol $\sim$ does not depend on $\pi$ or the integer $n$. Here, the inequality $\gtrsim$ uses bounded backtracking \cite{Cooper} as in \cite{BH92}, working in the graph $\overline{\mc{G}}$.

    \Cref{eq:total_length_estimate} shows that, for any element $g\in G$ not conjugate into one of the subgroups $G_1,\dots,G_k$, either the growth rate $\big[\,\|\phi^n(g)\|\,\big]$ is bounded, or we have:
    \begin{equation}\label{eq:pw_ineq}
        \big[\|\phi^n(g)\|\big]\sim \big[|\varphi^n(g)|\big] \preceq \sum_{i=1}^k\overline{\mc{O}}_{\rm top}(\varphi_i) + \big[\sum_{j=0}^n N_j\lambda^{n-j}\big] .
    \end{equation}
    In particular, we claim that this implies the following equality in $\mf{G}$: 
    \begin{equation}\label{eq:O_estimate}
        \overline{\mc{O}}_{\rm top}(\varphi) \sim \sum_{i=1}^k\overline{\mc{O}}_{\rm top}(\varphi_i) + \big[\sum_{j=0}^n N_j\lambda^{n-j}\big] .
    \end{equation}
    The inequality $\preceq$ is immediate from the definition of $\overline{\mc{O}}_{\rm top}(\varphi)$ and \Cref{eq:pw_ineq}. Conversely, it is clear that $\overline{\mc{O}}_{\rm top}(\varphi) \succeq \overline{\mc{O}}_{\rm top}(\varphi_i)$ for all $i$, while it follows from \Cref{eq:total_length_estimate} that we also have $\overline{\mc{O}}_{\rm top}(\varphi) \succeq \sum_{j=0}^n N_j\lambda^{n-j}$. Thus, $\overline{\mc{O}}_{\rm top}(\varphi)$ is (coarsely) bounded below by the sum of these rates.

    The entire discussion up to this point was completely general: we have not made use of any assumptions on the $G_i$ other than finite generation. 
    
    From now on assume that, for each $i$, either $\varphi_i\in\Aut(G_i)$ is $(\lambda_i,p_i)$--docile for some $\lambda_i>1$ and $p_i\in\N$, or the growth rate $\overline{\mc{O}}_{\rm top}(\varphi_i)$ is sub-polynomial (in which case, we set $\lambda_i:=1$ for convenience). We proceed to discuss the claims in the various parts of the proposition.

    \smallskip
    {\bf Part~(1).} Recalling \Cref{rmk:tame_implies_sum-stable} and the definition of $N_n$, we have
    \begin{equation}\label{eq:N_n_estimate}
        [N_n]\preceq \sum_{i=1}^k\overline{\mc{O}}_{\rm top}(\varphi_i) \preceq [n^p\lambda_*^n],
    \end{equation}
    where $\lambda_*:=\max\{\lambda_1,\dots,\lambda_k\}$, and $p$ is defined as the maximum of the $p_j$ such that $\lambda_j=\lambda_*$. Consequently, \Cref{eq:O_estimate} yields:
    \begin{itemize}
        \item if $\lambda_*>\lambda$, then $\overline{\mc{O}}_{\rm top}(\varphi) \sim \sum_{i=1}^k\overline{\mc{O}}_{\rm top}(\varphi_i)$ and the latter is $(\lambda_*,p)$--tame;
        \item if $\lambda_*\leq\lambda$, then $\overline{\mc{O}}_{\rm top}(\varphi)\sim\sum_{i=1}^k\overline{\mc{O}}_{\rm top}(\varphi_i)+\big[\sum_{j=0}^n N_j\lambda^{n-j}\big]$, where the latter growth rate satisfies $[\lambda^n]\preceq  \big[\sum_{j=0}^n N_j\lambda^{n-j}\big]\preceq [n^{p+1}\lambda^n]$. In particular, rewriting this growth rate as $\big[\lambda^n\cdot\sum_{j=0}^n N_j\lambda^{-j}\big]$, 
        we see that it is $(\lambda,p+1)$--tame.
    \end{itemize}

    We are only left to show that $\varphi$ is sound, namely that we have $\overline{\mc{O}}_{\rm top}(\varphi)\preceq\overline{\mf{o}}_{\rm top}(\phi)$. If $g\in G$ is not conjugate into any of the $G_i$, and its conjugacy class does not have finite $\phi$--orbit, then \Cref{eq:pw_ineq} shows that $|\varphi^n(g)|\sim\|\phi^n(g)\|\preceq\overline{\mf{o}}_{\rm top}(\phi)$. If instead $g$ is conjugate into some $G_i$, then $|\varphi^n(g)|\preceq\overline{\mc{O}}_{\rm top}(\varphi_i)\sim\overline{\mf{o}}_{\rm top}(\phi_i)\preceq\overline{\mf{o}}_{\rm top}(\phi)$. In conclusion, we have $|\varphi^n(g)|\preceq\overline{\mf{o}}_{\rm top}(\phi)$ for all $g\in G$, which yields $\overline{\mc{O}}_{\rm top}(\varphi)\preceq\overline{\mf{o}}_{\rm top}(\phi)$ as required.

    \smallskip
    {\bf Part~(2).} If $\lambda_*<\lambda$, then \Cref{eq:N_n_estimate} shows that $\sum_{j=0}^n N_j\lambda^{n-j}\sim\lambda^n$. Thus, \Cref{eq:pw_ineq} implies that $\|\phi^n(g)\|\sim\lambda^n$ for all elements $g\in G$ not conjugate into any $G_i$, and whose conjugacy class is not preserved by a power of $\phi$. We thus have $[\lambda^n]\in\mf{g}(\phi)$ in this case. If instead $\lambda_*\geq\lambda$, then the assumptions of part~(2) imply that $[\lambda_*^n]\in\mf{g}(\phi_i)\sq\mf{g}(\phi)$ for some index $i$. Either way, recalling that $\mu=\max\{\lambda_*,\lambda\}$, we have $[\mu^n]\in\mf{g}(\phi)$.

    \smallskip
    {\bf Part~(3).} Suppose that there exist a finite set $\Lambda\sq\R_{>1}$ and $P\in\N$ such that, for all indices $i$ and all representatives $\varphi_i'\in\Aut(G_i)$ of $\phi_i\in\Out(G_i)$, each growth rate in the union $\mc{G}(\varphi_i')\cup\mf{g}(\phi_i)$ is either sub-polynomial or equal to $[n^a\nu^n]$, for some $\nu\in\Lambda$ and some integer $0\leq a\leq P$.
    
    Our goal is to precisely estimate the growth rate $\big[\,\|\phi^n(g)\|\,\big]$ for all elements $g\in G$ not conjugate into any $G_i$ and such that $\|\phi^n(g)\|\not\sim 1$. \Cref{eq:total_length_estimate} shows that $\big[\,\|\phi^n(g)\|\,\big]$ is a sum of finitely many growth rates in $\bigcup_i\mc{G}(\varphi_i)$ with the growth rate $\big[\,\sum_{j=0}^nN_j\lambda^{n-j}\,\big]$. Since a finite sum of pure growth rates equals the fastest among them, it suffices to show that either $\sum_{j=0}^nN_j\lambda^{n-j}\sim\lambda^n$, or $\sum_{j=0}^nN_j\lambda^{n-j}\sim n^a\nu^n$ for some $\nu\in\Lambda$ and an integer $0\leq a\leq P+1$.

    Recall that the integer $N_{j+1}$ is defined as the average word length of a uniformly bounded number of elements of the form 
    \[ a\psi(a)\dots\psi^{j-1}(a)\cdot\psi^j(c)\cdot\psi^{j-1}(e)\dots\psi(e)e, \]
    with $a,c,e\in G_i$ for some index $i$. We will thus need the following observation.

    \smallskip
    {\bf Claim.} \emph{Consider some $a,c,e\in G_i$ and set $u_n:=a\psi(a)\dots\psi^{n-1}(a)\cdot \psi^n(c)\cdot\psi^{n-1}(e)\dots\psi(e)e$. Then the growth rate $\big[\,|u_n|+|u_{n+1}|\,\big]$ is either sub-polynomial or equal to $[n^a\nu^n]$ for some $\nu\in\Lambda$ and an integer $0\leq a\leq P$.}

    \smallskip\noindent
    \emph{Proof of claim.}
    Setting for simplicity $e_n:=\psi^{n-1}(e)\dots\psi(e)e$, observe that we have
    \[ u_n^{-1}u_{n+1}=e_n^{-1}\psi^n\big(c^{-1}a\psi(c)e\big)e_n .\]
    Recall that $\psi|_{G_i}$ coincides with the automorphism $\varphi_i\in\Aut(G_i)$, and let $\eta\in\Aut(G_i)$ be the automorphism defined by $\eta(x)=e^{-1}\varphi_i(x)e$. For every $n\geq 1$, we have $\eta^n(x)=e_n^{-1}\varphi_i^n(x)e_n$. Thus, setting $w:=c^{-1}a\psi(c)e$, we can rewrite the above equality simply as
    \[ u_n^{-1}u_{n+1}=\eta^n(w). \]
    As $\eta$ is in the same outer class as $\varphi_i$, the hypotheses of part~(3) guarantee that every growth rate in $\mc{G}(\eta)$ is sub-polynomial or equal to $[n^a\nu^n]$ for $\nu\in\Lambda$ and $0\leq a\leq P$. Thus, the same is true of the growth rate $\big[\,|u_n^{-1}u_{n+1}|\,\big]$.

    If $\big[\,|u_n^{-1}u_{n+1}|\,\big]$ is sub-polynomial, then a telescopic argument shows that $\big[\,|u_n|\,\big]$ is sub-polynomial. Suppose instead that $|u_n^{-1}u_{n+1}|\sim n^a\nu^n$. Then, we similarly get $|u_n|\preceq\sum_{j=1}^nj^a\nu^j\preceq n^a\nu^n$,
    while the triangle inequality yields $|u_n|+|u_{n+1}|\succeq |u_n^{-1}u_{n+1}|\sim n^a\nu^n$. In conclusion, we obtain the equality $|u_n|+|u_{n+1}|\sim n^a\nu^n$, as desired.
    \hfill$\blacksquare$

    \smallskip
    Now, the claim implies that the growth rate $\big[N_n+N_{n+1}\big]$ is sub-polynomial or equal to $[n^a\nu^n]$ with $\nu\in\Lambda$ and $0\leq a\leq P$. Carrying out the sum two terms at a time, this shows that $\lambda^n\cdot\sum_{j=0}^n N_j\lambda^{-j}$ is $\sim\lambda^n$ if $\nu<\lambda$, while it is $\sim n^a\nu^n$ if $\nu>\lambda$, and finally $\sim n^{a+1}\nu^n$ if $\nu=\lambda$.
    
    This concludes the proof of part~(3) and of the entire proposition.
\end{proof}

The restriction to fully irreducible automorphisms in \Cref{prop:growth_estimate_fully_irreducible} was meant to simplify the already technical proof. However it can be easily removed, which we do in the next corollary.

\begin{cor}\label{cor:growth_free_product}
    Let $G$ be finitely generated and infinitely ended, with freely indecomposable free factors $H_i$. Consider $\phi\in\Out(G)$, let $\phi_i\in\Out(H_i)$ be the restrictions of $\phi$, and let $\varphi_i\in\Aut(H_i)$ be representatives of the $\phi_i$. For each $i$, suppose that either $\overline{\mc{O}}_{\rm top}(\varphi_i)$ is sub-polynomial, or $\varphi_i$ is $(\lambda_i,p_i)$--docile for some $\lambda_i>1$ and $p_i\in\N$. Then all the following hold.
    \begin{enumerate}
        \item Either $\phi$ is represented by some $\varphi\in\Aut(G)$ with sub-polynomial $\overline{\mc{O}}_{\rm top}(\varphi)$, or $\phi$ is $(\mu,q)$--docile for some $\mu>1$ and $q\in\N$, where $\mu$ is either $\max_i\lambda_i$ or a larger Perron number.
        \item If we have $[\lambda_i^n]\in\mf{g}(\phi_i)$ for all indices $i$ such that $\lambda_i=\mu$, then $[\mu^n]\in\mf{g}(\phi)$. 
    \end{enumerate}
\end{cor}
\begin{proof}
    Write $G=H_1\ast\dots\ast H_{k'}\ast F_{m'}$ where the $H_i$ are freely indecomposable. We can assume that $k'\geq 1$ since all automorphisms of free groups satisfy the thesis (see e.g.\ \cite{Levitt-GAFA}). Defining $\mc{F}'$ as the set of $G$--conjugates of the $H_i$, it is immediate that $\mc{F}'$ is a $\phi$--invariant factor system. Up to raising $\phi$ to a power, we have $\phi\in\Out(G,\mc{F}')$. 
    Up to further raising $\phi$ to a power, an iterated application of \Cref{lem:enlarging_factor_system} yields a factor system $\mc{F}>\mc{F}'$ such that $\phi\in\Out(G,\mc{F})$ and such that $\phi$ is fully irreducible for $(G,\mc{F})$. 

    Arguing by induction on the Grushko rank of $G$, we can assume that the corollary holds for the restriction of $\phi$ to each element of $\mc{F}$. If $\mc{F}$ is non-sporadic, then the corollary was the content of \Cref{prop:growth_estimate_fully_irreducible}. If instead $\mc{F}$ is sporadic, then $G$ admits a $\phi$--invariant free splitting whose vertex groups are the elements of $\mc{F}$, and in this case the corollary follows from \Cref{prop:growth_GOGs}.
\end{proof}

\bibliography{./mybib}

\begin{thebibliography}{Lym22b}

\bibitem[BFH00]{BFH1}
Mladen Bestvina, Mark Feighn, and Michael Handel.
\newblock The {T}its alternative for {${\rm Out}(F_n)$}. {I}. {D}ynamics of
  exponentially-growing automorphisms.
\newblock {\em Ann. of Math. (2)}, 151(2):517--623, 2000.

\bibitem[BFH05]{BFH2}
Mladen Bestvina, Mark Feighn, and Michael Handel.
\newblock The {T}its alternative for {${\rm Out}(F_n)$}. {II}. {A} {K}olchin
  type theorem.
\newblock {\em Ann. of Math. (2)}, 161(1):1--59, 2005.

\bibitem[BG10]{Bridson-Groves}
Martin~R. Bridson and Daniel Groves.
\newblock The quadratic isoperimetric inequality for mapping tori of free group
  automorphisms.
\newblock {\em Mem. Amer. Math. Soc.}, 203(955):xii+152, 2010.

\bibitem[BH92]{BH92}
Mladen Bestvina and Michael Handel.
\newblock Train tracks and automorphisms of free groups.
\newblock {\em Ann. of Math. (2)}, 135(1):1--51, 1992.

\bibitem[CHHL]{CHHL}
R\'{e}mi Coulon, Arnaud Hilion, Camille Horbez, and Gilbert Levitt.
\newblock In preparation.

\bibitem[Coo87]{Cooper}
Daryl Cooper.
\newblock Automorphisms of free groups have finitely generated fixed point
  sets.
\newblock {\em J. Algebra}, 111(2):453--456, 1987.

\bibitem[Cou22]{Coulon}
R\'{e}mi Coulon.
\newblock Examples of groups whose automorphisms have exotic growth.
\newblock {\em Algebr. Geom. Topol.}, 22(4):1497--1510, 2022.

\bibitem[CT94]{CT94}
Donald~J. Collins and Edward~C. Turner.
\newblock Efficient representatives for automorphisms of free products.
\newblock {\em Michigan Math. J.}, 41(3):443--464, 1994.

\bibitem[Fio25]{Fio11}
Elia Fioravanti.
\newblock Growth of automorphisms of virtually special groups.
\newblock {\em arXiv:2501.12321}, 2025.

\bibitem[FM15]{FM15}
Stefano Francaviglia and Armando Martino.
\newblock Stretching factors, metrics and train tracks for free products.
\newblock {\em Illinois J. Math.}, 59(4):859--899, 2015.

\bibitem[GH22]{GH22}
Vincent Guirardel and Camille Horbez.
\newblock Boundaries of relative factor graphs and subgroup classification for
  automorphisms of free products.
\newblock {\em Geom. Topol.}, 26(1):71--126, 2022.

\bibitem[GL17]{GL-JSJ}
Vincent Guirardel and Gilbert Levitt.
\newblock J{SJ} decompositions of groups.
\newblock {\em Ast\'{e}risque}, (395):vii+165, 2017.

\bibitem[Lev09]{Levitt-GAFA}
Gilbert Levitt.
\newblock Counting growth types of automorphisms of free groups.
\newblock {\em Geom. Funct. Anal.}, 19(4):1119--1146, 2009.

\bibitem[Lym22a]{Lyman-CTs}
Rylee~A. Lyman.
\newblock {CT}s for free products.
\newblock {\em arXiv:2203.08868}, 2022.

\bibitem[Lym22b]{Lyman-train-track}
Rylee~A. Lyman.
\newblock Train track maps on graphs of groups.
\newblock {\em Groups Geom. Dyn.}, 16(4):1389--1422, 2022.

\bibitem[Ols97]{Olsh2}
A.~Yu. Olshanski\u{\i}.
\newblock On the distortion of subgroups of finitely presented groups.
\newblock {\em Mat. Sb.}, 188(11):51--98, 1997.

\bibitem[Ols99]{Olsh1}
A.~Yu. Olshanski\u{\i}.
\newblock Distortion functions for subgroups.
\newblock In {\em Geometric group theory down under ({C}anberra, 1996)}, pages
  281--291. de Gruyter, Berlin, 1999.

\bibitem[RS97]{RS97}
Eliyahu Rips and Zlil Sela.
\newblock Cyclic splittings of finitely presented groups and the canonical
  {JSJ} decomposition.
\newblock {\em Ann. of Math. (2)}, 146(1):53--109, 1997.

\bibitem[Thu88]{Thurston-BAMS}
William~P. Thurston.
\newblock On the geometry and dynamics of diffeomorphisms of surfaces.
\newblock {\em Bull. Amer. Math. Soc. (N.S.)}, 19(2):417--431, 1988.

\end{thebibliography}
\bibliographystyle{alpha}

\end{document}